\documentclass[12pt, reqno]{amsart}
 \usepackage[foot]{amsaddr}

\usepackage{amsmath}
\usepackage{amsthm}
\usepackage{amsfonts}
\usepackage{amssymb}
\usepackage{verbatim}
\usepackage{graphicx}
\usepackage[margin=1.0in]{geometry}
\usepackage{url}
\usepackage{enumerate}
\usepackage[pdftex,bookmarks=true]{hyperref}
\usepackage[normalem]{ulem}
\usepackage[usenames, dvipsnames]{color}

\theoremstyle{plain}
\newtheorem{theorem}{Theorem}[section]
\newtheorem{conjecture}[theorem]{Conjecture}

\newtheorem{proposition}[theorem]{Proposition}

\newtheorem{lemma}[theorem]{Lemma}

\theoremstyle{definition}
\newtheorem{definition}[theorem]{Definition}

\newcommand\C{{\mathbb{C}}}

\newcommand\R{{\mathbb{R}}}
\renewcommand\P{{\mathbb{P}}}
\newcommand\E{{\mathbb{E}}}  
\newcommand\Var{\mathrm{Var}}

\newcommand{\ep}{\varepsilon}

\newcommand{\ba}{\[\begin{aligned}}
\newcommand{\ea}{\end{aligned}\]}

\let\oldtocsection=\tocsection

\let\oldtocsubsection=\tocsubsection

\let\oldtocsubsubsection=\tocsubsubsection

\renewcommand{\tocsection}[2]{\hspace{0em}\oldtocsection{#1}{#2}}
\renewcommand{\tocsubsection}[2]{\hspace{1em}\oldtocsubsection{#1}{#2}}
\renewcommand{\tocsubsubsection}[2]{\hspace{2em}\oldtocsubsubsection{#1}{#2}}
\begin{document}
\title[Convergence of higher derivatives]{Convergence of higher derivatives of random polynomials with independent roots}

\author{J\"urgen Angst $^1$, Oanh Nguyen $^2$ and Guillaume Poly $^3$}


\address{$^1$ Univ Brest, CNRS UMR 6205, Laboratoire de Mathematiques de Bretagne Atlantique, F--29200 Brest, France} 
\address{$^2$ Division of Applied Mathematics, Brown University, Providence, RI 02906, USA}
\address{$^3$ Nantes Universit\'e, CNRS, Laboratoire de Math\'ematiques Jean Leray, LMJL,
F--44000 Nantes, France.}

\email{jurgen.angst@univ-brest.fr}
\email{oanh$\_$nguyen1@brown.edu}
\email{guillaume.poly@univ-nantes.fr}

\maketitle

 \begin{abstract}
 	Let $\mu$ be a probability measure on $\C$, and let $P_n$ be the random polynomial whose zeros are sampled independently from $\mu$. We study the asymptotic distribution of zeros of high-order derivatives of $P_n$. We show that, for large classes of measures $\mu$, the empirical distribution of zeros of the $k$-th derivative converges back to $\mu$ for all derivative orders $k=o(n/\log n)$. This includes all discrete measures and a broad family of measures satisfying a mild dimension-nondegeneracy condition. 
 	
 	We further establish a robustness result showing that, for arbitrary $\mu$, even after adding a vanishing proportion of roots drawn from a dimension-nondegenerate perturbation, the derivative zero measures still converge back to $\mu$. These results break the previously known logarithmic barrier on the order of differentiation and demonstrate that the limiting root distribution is preserved under differentiation of order growing nearly linearly with the degree.
  \end{abstract}
 
 \maketitle

 
 \section{Introduction}
 In the deterministic setting, the relation between the roots of a polynomial and those of its derivatives remains surprisingly intricate. Classical results such as the Gauss–Lucas theorem give only broad geometric constraints, and several fundamental questions about the behavior of critical points are still open; see, for example, \cite{Tao22, GLSv07}. Because the deterministic theory is difficult to push much further, it is natural to turn to randomness as a way to uncover underlying patterns. In this context, differentiation blends the influence of all zeros, so the critical points capture the overall shape of the root configuration. Random polynomials therefore provide a flexible setting in which this interaction becomes amenable to analysis and connects naturally with ideas from complex analysis and potential theory. Understanding the asymptotic distribution of critical points thus becomes both tractable and conceptually compelling once the random root configuration is specified. Let us specify the model we consider throughout the article and described the main related literature.
 
 Let $\mu$ be a probability measure on $\mathbb{C}$, and let $\xi_1, \xi_2, \ldots$ be i.i.d.\ random variables with distribution $\mu$.  
 For each $n \geq 1$, define the random polynomial
 \[
 P_n(z) = \prod_{j=1}^n (z - \xi_j).
 \]
 By the Law of Large Numbers, the empirical measure $\mu_n:=\frac1n \sum_{i=1}^{n} \delta_{\xi_i}$ converges almost surely to $\mu$ as $n \to \infty$. Pemantle and Rivin~\cite{pemantle2013distribution} conjectured that the empirical distribution of the critical points of $P_n$ would also converge to $\mu$.  
 Specifically, if
 \[
 \mu_{P_n'} := \frac{1}{n-1} \sum_{z \in \mathbb{C} : P_n'(z) = 0} \delta_z
 \]
 denotes the empirical measure of $P_n'$, then $\mu_{P_n'}$ converges to $\mu$ as $n \to \infty$.  
 They proved this under the additional assumption that $\mu$ has finite $1$-energy, and Subramanian~\cite{subramanian2012distribution} established the result when $\mu$ is supported on the unit circle.  
 
 \medskip
 A major advance was achieved by Kabluchko~\cite{Kabluchko2015}, who confirmed the conjecture in full generality for all probability measures $\mu$, asserting that $\mu_{P_n'}$ converges in probability to $\mu$  in the space of complex probability measures equipped with the topology of the convergence in distribution.  Answering a question raised by Kabluchko, together with Malicet, the first and third author \cite{angst2024almost} proved that in fact, $\mu_{P_n'}$ converges almost surely to $\mu$ in the same topology of convergence in distribution. We note that the former is equivalent to saying that for all bounded continuous test functions $\varphi$,
 \[
 \int_{\mathbb{C}} \varphi(x) \, d\mu_{P_n'}(x)
 \;\xrightarrow[n\to\infty]{\mathbb{P}}\;
 \int_{\mathbb{C}} \varphi(x) \, d\mu(x),
 \]
 while the latter is equivalent to  
 \[
 \int_{\mathbb{C}} \varphi(x) \, d\mu_{P_n'}(x)
 \;\xrightarrow[n\to\infty]{\text{a.s.}}\;
 \int_{\mathbb{C}} \varphi(x) \, d\mu(x).
 \]
 
 \vspace{5mm}
 The study of higher derivatives followed naturally.  
 For $1 \leq k < n$, define
 \[
 \mu_{P_n^{(k)}} := \frac{1}{n-k} \sum_{z \in \mathbb{C} : P_n^{(k)}(z) = 0} \delta_z.
 \]

 Byun, Lee, and Reddy~\cite{ByunLeeReddy2022} showed that for each fixed $k$, one still has $\mu_{P_n^{(k)}}$ converges to $\mu$ in probability.
 O’Rourke and Steinerberger~\cite{ORourkeSteinerberger2021} proposed a broader conjectural picture for large derivatives: for $t \in (0,1)$, the empirical measure of the zeros of $P_n^{(\lfloor tn \rfloor)}$ converges to a deterministic limit $\mu_t$, whose logarithmic potential should solve a certain PDE.  
 This PDE has connections to earlier work ~\cite{alazard2022dynamics, kiselev2022flow} and, in special cases, the conjecture has been verified when $\mu$ has real support~\cite{hoskins2023dynamics} using free probability.  
 In the case $t=0$, their prediction recovers $\mu_t = \mu$, leading to the following conjecture.
 \begin{conjecture}\label{conj}
 	For all probability measures $\mu$ on the complex plane, for any sequence $k_n$ satisfying $k_n=o(n)$, it holds that almost surely, $\mu_{P_n^{(k_n)}}$ converges  to $ \mu$.
 \end{conjecture}

 \vspace{5mm}
 The current best result is due to Michelen--Vu who showed in  \cite{michelen2024zeros} that convergence in probability holds for all $k:=k_n\le \frac{\log n}{5\log\log n}$ and in \cite{michelen2024almost} that almost sure convergence holds for all fixed $k$. To see that $\log n$ is a natural barrier for this kind of result, we first recall the standard distributional identity $\frac{1}{2\pi}\Delta \log |z-\xi| = \delta_\xi$ which connects the logarithmic potential $\frac{1}{n}\log |P_n(z)|$ of $P_n$ to its empirical measure:  
 $$\frac{1}{2\pi n} \Delta \log |P_n(z)| = \mu_{P_n}.$$
 Now, if we look at the quotient
 \begin{equation}\label{eq:S:xi}
 	S_n =	S_{k, n}(z):= \frac{P^{(k)}_n(z)}{k!P_n(z)},
 \end{equation}
 the above relationship gives
 \[
 \frac{1}{2\pi n} \Delta \log \left| S_n(z) \right|
 = \frac{1}{n} \sum_{\xi \in \mathbb{C} : P_n^{(k)}(\xi) = 0} \delta_\xi
 - \frac{1}{n} \sum_{\xi \in \mathbb{C} : P_n(\xi) = 0} \delta_\xi
 = \frac{n-k}{n} \,\mu_{P_n^{(k)}} - \mu_{P_n}.
 \]
 By the Law of Large Numbers, the empirical measure \(\mu_{P_n}\) converges to \(\mu\) as \(n \to \infty\).
 Consequently, when \(k = o(n)\), establishing the convergence \(\mu_{P_n^{(k)}} \to \mu\) reduces to proving that
 \(\frac{1}{n} \log |S_n(z)|\) tends to \(0\). Controling the growth of $\log |S_n(z)|$ amounts to proving an upper bound on $|S_n(z)|$ and a lower bound on $|S_n(z)|$, which can be quantified as proving
 \begin{equation}
 	\P(|S_n(z)|\le e^{-\ep n}) \notag
 \end{equation}
 to be sufficiently small as $n\to\infty$.
 
 Moreover, rewriting $S_n$ as 
 $$S_n(z)= \sum_{i_1<i_2<\dots<i_k} Y_{i_1}\dots Y_{i_k}$$
 where $Y_i = \frac{1}{z-\xi_i}$ are iid, we see that $S_n$ is a degree $k$ polynomials of the random variables $Y_i$. This directly connects this problem to the anticoncentration problem for degree-$k$ polynomials of independent random variables. 
 
 The latter has been studied for its intrinsic interest and has applications in various areas, including random matrix theory and subgraph counting; see, for instance, Meka--Nguyen--Vu~\cite{meka2015anti}, Kwan--Sudakov--Tran~\cite{kwan2019anticoncentration}, Alon--Hefetz--Krivelevich--Tyomkyn\\~\cite{alon2020edge}, and Fox--Kwan--Sauermann~\cite{fox2021combinatorial}. The above list is far from exhaustive.
 We note that $\log n$ is a natural barrier for the degree, as can be seen in the following example of a degree $k$ polynomials on $n=mk$ random variables $(Y_{ij})$ 
 \begin{equation}\label{eq:counter}
 	S = \sum_{i=1}^{m} Y_{i1}\dots Y_{ik}
 \end{equation}
 where $i=1, \dots, m$, $j=1, \dots, k$ and $Y_{ij}$ are iid with any distribution having $\P(Y_{ij} = 0) =q$ for some $q\in (0, 1)$  (say $q=1/3$ and $Y_{ij}$ is uniform in $\{0, \pm 1\}$) and $m = (1-q)^{-k}$. This corresponds to $n = k(1-q)^{-k}$ or $k=\Theta(\log n)$. We have
 \begin{equation}\label{key}
 	\P(S = 0)\ge \P(Y_{i1}\dots Y_{ik} = 0)^{m} = (1-(1-q)^{k})^{(1-q)^{-k}} \approx e^{-1}\notag
 \end{equation}
 which clearly does not decay to $0$.

Breaking this natural barrier, our main theorem shows that for non-pathological measures $\mu$, Conjecture \ref{conj} holds all the way to $k = o(n/\log n)$. Before stating the theorem, let us recall the notion of local dimension of a measure $\mu$ at a given point $x\in \C$:
\[
\underline{\dim}_{\mu}(x)
:= \liminf_{r\to 0}
\frac{\log \mu(D(x,r))}{\log r}.
\]
\begin{definition}[Dimension-nondegenerate measures] \label{def:dim}
	 We say that a measure $\mu$ is {\it dimension-nondegenerate} if there exists a measure $\nu\le \mu$ (namely $\nu(A)\le \mu(A)$ for all Borel sets $A$) and a set $E\subset \C$ of positive measure $\nu(E)>0$ such that for all $x\in E$, 
	 $$\underline{\dim}_{\nu}(x)>0.$$
\end{definition}

 \begin{theorem}[Main theorem]\label{thm:main}  Let $\mu$ be an element in the space of probability measures on $\C$ equipped with the topology of the convergence in distribution. 
	\begin{enumerate} 
		\item \label{cond:disc}[Discrete measures] If $\mu$ is discrete, then for any sequence $k_n=o(n)$, almost surely with respect to $\P$, the sequence of measures $\mu_{P_n^{(k_n)}}$ converges to $\mu$ as $n\to \infty$.
		\item \label{cond:nondeg} [Dimension-nondegenerate measures] Assume that $\mu$ is dimension-nondegenerate then for any sequence $k_n=o(\frac{n}{\log n})$, almost surely with respect to $\P$, the sequence of measures $\mu_{P_n^{(k_n)}}$ converges to $\mu$ as $n\to \infty$. 
		\end{enumerate}
\end{theorem} 
 
 Observe that if a measure $\mu$ contains a dimension-nondegenerate component, that is,
 \[
 \mu = (1-\varepsilon)\mu_0 + \varepsilon \nu, \quad \varepsilon>0,
 \]
 with $\nu$ dimension-nondegenerate, then $\mu$ itself is dimension-nondegenerate. Below, we list several examples of dimension-nondegenerate measures; any measure that includes such a component is automatically dimension-nondegenerate as well.
 
 \begin{enumerate}
   
 	\item \emph{Measures not entirely singular with respect to the Lebesgue measure.}
 	In the case that $\mu$ is not entirely singular, by the classical Radon-Nikodym's theorem, there exists a nonnegative, integrable function $f\not\equiv 0$ such that $\mu = f(z)\,dm_2 + \mu_s$ where $m_2$ is the Lebesgue measure on $\C$ and $\mu_s$ is singular with respect to $m_2$. There exists an $M>0$ such that $f_M:= f \text{1}_{f\le M} \not\equiv 0$. We then let $\nu = f_M dm_2$ and $E=\C$. For all $x\in \C$, $r>0$, we then have 
 	$$\nu(D(x, r))\le \pi M r^{2}$$
 	and as $\log r<0$ for $r<1$, we get
 	$$ \liminf_{r\to 0}\frac{\log \nu(D(x,r))}{\log r} \ge  \liminf_{r\to 0} \frac{\log \pi M r^{2}}{\log r}\ge 2>0$$
 	proving dimension-nondegeneracy.
 	
 	\item \emph{Measures satisfying a local Frostman condition on a positive-mass set.}
 	If there exist constants $s>0$, $C>0$, and a Borel set $E\subset\C$
 	with $\mu(E)>0$ such that
 	\[
 	\mu(D(x,r)) \le C r^{s}
 	\quad \text{for all } x\in E \text{ and all sufficiently small } r>0,
 	\]
 	then $\underline{\dim}_{\mu}(x)\ge s$ for all $x\in E$.

 	\item \emph{Measures supported on smooth curves.}
 	If $\mu$ is absolutely continuous with respect to arc-length measure
 	on a $C^1$ curve in $\C$, then $\underline{\dim}_{\mu}(x)=1$ for
 	$\mu$-a.e.\ $x$, hence $\mu$ is dimension-nondegenerate (see for example \cite{mattila1999geometry}). One of the most studied examples in this case is $\mu$ being the uniform measure on the unit circle.
 \end{enumerate}

  In the next theorem, although it is formally covered by
  Theorem~\ref{thm:main}, we present it separately because it is
  proved using fundamentally different methods and highlights a
  complementary mechanism of anti-concentration. Presenting different approaches illustrates the robustness of
  the phenomenon and provides additional insight into how zeros of
  high-order derivatives stabilize under different structural assumptions
  on the measure. We say that $\mu$ satisfies the Doeblin condition if there exist constants $c_0, r_0\in (0, \infty)$ and $z_0\in \C$ such that for any Borel measurable set $A\in B(z_0, r_0)$, it holds that
 \begin{equation}
 	\mu(A)\ge c_0 \text{Leb}(A).\label{cond:doeblin}
 \end{equation}
 In other words, a measure $\mu$ satisfies the Doeblin condition if it includes a component that is uniform on some disk.
 
 Recall the Cauchy-Stieljes transform $g_{\mu}$ of the measure $\mu$
 \begin{equation}
 	g_{\mu}(z):= \int_{\C} \frac{1}{z-u}d\mu(u).\label{def:CS}
 \end{equation}

 \begin{theorem}\label{thm:2} Let $\mu$ be an element in the space of probability measures on $\C$ equipped with the topology of the convergence in distribution. 
 	\begin{enumerate} 
 		\item \label{cond:0} [Doeblin condition] Assume that $\mu$ satisfies the Doeblin condition with parameters $c_0, r_0, z_0$. Then for any sequence $k_n=o(\frac{n}{\log n})$, almost surely with respect to $\P$, the sequence of measures $\mu_{P_n^{(k_n)}}$ converges to $\mu$ as $n\to \infty$. 
 		\item \label{cond:mean} [Generic measures] Assume that for Lebesgue-almost all $z\in \C$, $g_\mu(z)\neq 0$  and\\ $\int_{\C} \frac{1}{|z-u|^2}d\mu(u)<\infty$. Then for all $k = o(\sqrt n)$ the sequence of measures $\mu_{P_n^{(k_n)}}$ converges to $\mu$ in probability.
 	\end{enumerate}
 \end{theorem}

 In the general setting, for {\it any} probability measures $\mu$, our next theorem shows that if we perturb it by adding a small dimension-nondegenerate noise, namely we sample the roots from the perturb measure $\mu_n:=(1-\alpha_n)\mu + \alpha_n\nu$ where $\alpha_n = o(1)$ and $\nu$ is dimension-nondegenerate, then the roots of the corresponding polynomial $P_n^{(k)}$ converges to $\mu$.  This result can be viewed as complimentary to \cite[Corollary 1.3]{ByunLeeReddy2022}

 \begin{theorem}\label{thm:perturb}
 	Let $\mu$ and $\nu$ be two probability measures on $\C$. Assume that $\nu$ is dimension-nondenegerate. Let $\alpha_n\in (0, 1)$ be a sequence of deterministic numbers satisfying $\lim \alpha_n = 0$ and $\lim \frac{\alpha_n}{\log n/n} =\infty$. Let $\xi_1, \dots, \xi_n$ be iid samples from  
 	$$\mu_n=(1-\alpha_n)\mu + \alpha_n\nu$$
 	and let 
 	$$P_n = \prod_{i=1}^{n} (z - \xi_i).$$
 	Let $k_n$ be a sequence of integers satisfying $k_n = o(\frac{n\alpha_n}{\log  n})$, then the sequence of measures $\mu_{P_n^{(k_n)}}$ converges to $\mu$ as $n\to \infty$ almost surely.
 \end{theorem}

 To interpret the theorem, in the perturbed setting, we see that most roots of $P_n$ are sampled from $\mu$, while only a vanishing fraction $\alpha_n n$ come from a dimension-nondegenerate component $\nu$. 
 By the Gauss--Lucas theorem, the zeros of $P_n^{(k_n)}$ are influenced by the logarithmic forces generated by all roots of $P_n$, including those coming from $\nu$. 
 Nevertheless, Theorem~\ref{thm:perturb} shows that these additional roots do not alter the limiting behavior: despite the pulling effect induced by $\nu$, the zeros of $P_n^{(k_n)}$ still converge to $\mu$ (see Figure \ref{fig:comparison}). 
 This robustness strongly suggests that the same convergence should hold even in the unperturbed case $\alpha_n\equiv 0$, where all roots are drawn from $\mu$.
 \begin{figure}[h!]
 	\centering
 	\includegraphics[width=\textwidth]{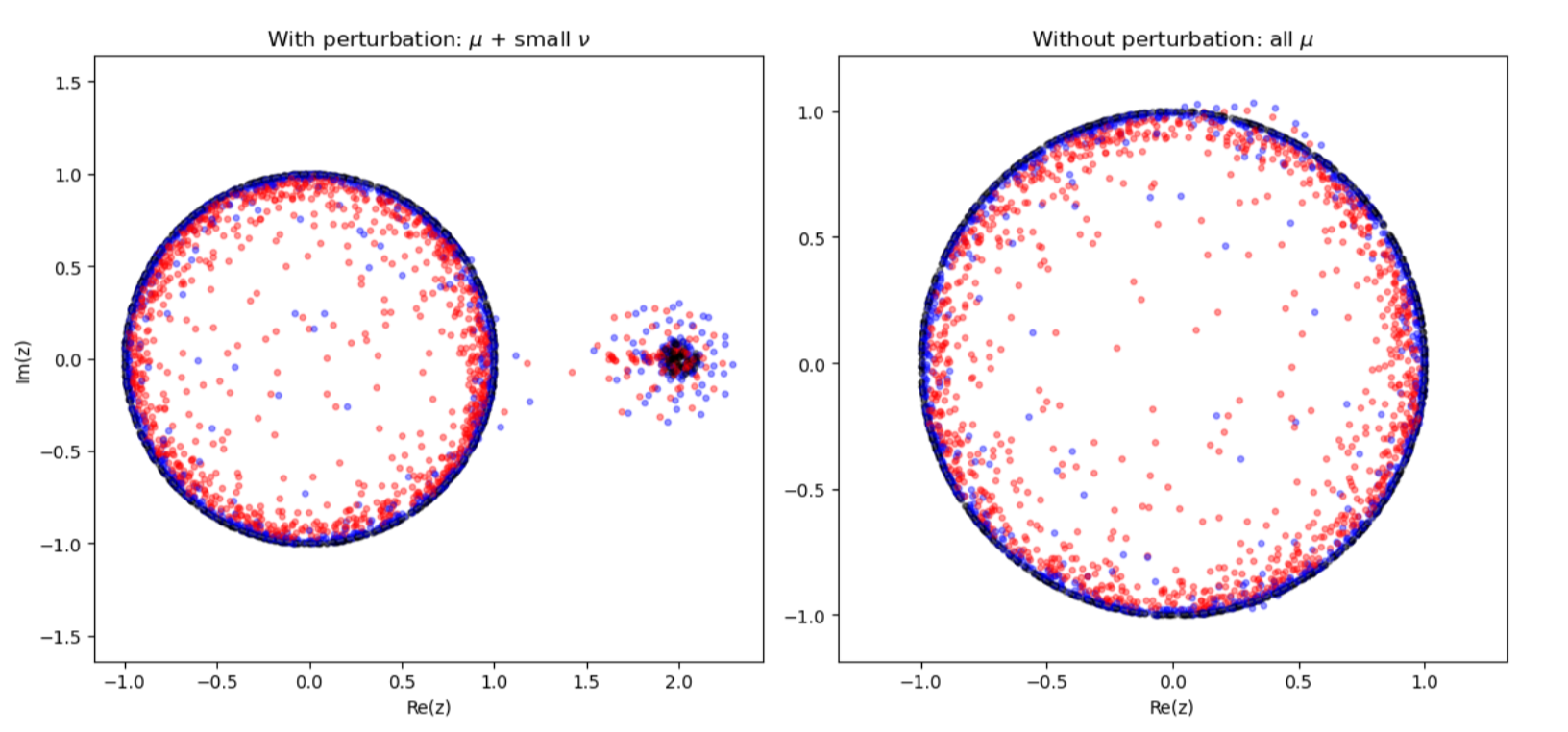}
 	\caption{Zeros of $P_n$ (black), $P_n'$ (blue), and $P_n^{(5)}$ (red) for ten independent samples.
 		Left: $100$ roots are drawn from $\mu$ (uniform on the unit circle) and $10$ from $\nu$ (uniform on a small disk of radius $0.1$ centered at $3$).
 		Right: all $110$ roots are drawn from $\mu$.
 	}
 	\label{fig:comparison}
 \end{figure}

 \section{Proof of Theorem \ref{thm:main}}
 To simplify notations, we will write $k=k_n$.
 \subsection{Discrete measures} Firstly, we consider the case that $\mu$ is supported on a set $\{z_1, z_2, \dots\}$, we note that a.s., the polynomial $P_n$ has root $z_i$ with multiplicity $N_i$ with $N_i/n\to \mu(z_i)$. Taking derivatives $k_n=o(n)$ times, $z_i$ remains a root with multiplicity $N_i-o(n)=(1+o(1)) \mu(z_i)n$. Therefore, a.s.,
 $$\liminf_{n\to\infty}\mu_{P_n^{(k)}}(z_i)\ge \mu(z_i).$$
 Since the sum over $i$ of the left-hand side is at most 1, which is the sum of the right-hand side, it must be the case that equality holds for all $i$ which in turns implies that 
 \begin{eqnarray*}
 	1-\liminf_{n\to\infty}\mu_{P_n^{(k)}}(z_i)&=& \limsup_{n\to\infty} \sum_{j\neq i}\mu_{P_n^{(k)}}(z_j)= 1-\mu(z_i)=\sum_{j\neq i}\mu(z_j)\\
 	&=&\sum_{j\neq i} \liminf \mu_{P_n^{(k)}}(j)\le \liminf \sum_{j\neq i} \mu_{P_n^{(k)}}(j).
 \end{eqnarray*}
 Therefore, 
 $\limsup_{n\to\infty} \sum_{j\neq i}\mu_{P_n^{(k)}}(z_j)= \liminf \sum_{j\neq i} \mu_{P_n^{(k)}}(j)= \lim \sum_{j\neq i} \mu_{P_n^{(k)}}(j)$
 and so
 $$\lim_{n\to\infty} \mu_{P_n^{(k)}}(z_i)=\liminf _{n\to\infty}\mu_{P_n^{(k)}}(z_i)= \mu(z_i)$$
 We conclude that $\mu_{P_n^{(k)}}$ converges to $\mu$ almost surely. 
 Therefore, for the rest of this section, we assume that $\mu$ is not discrete.

 \subsection{Dimension-nondegenerate measures}\label{sec:proof:mu}
Following the notations and proof strategy in \cite{angst2024almost}, we denote by 
 \[
 \mathcal{M} = \left\{ z \mapsto \frac{\alpha z + \beta}{\gamma z + \delta} 
 \,\middle|\, \alpha, \beta, \gamma, \delta \in \mathbb{C},\, \alpha\delta - \beta\gamma \neq 0 \right\}
 \]
 the set of invertible Möbius transformations of $\mathbb{C}$ and endow it with the measure $\lambda_{\mathcal{M}}$ inherited from 
 the Lebesgue measure $\lambda_{\mathbb{C}} \otimes \lambda_{\mathbb{C}} \otimes 
 \lambda_{\mathbb{C}} \otimes \lambda_{\mathbb{C}}$ on $\mathbb{C}^4$.
 
 Let $u$ be in $\mathcal M$. The classical Jensen's formula when applied to the meromorphic function $S_n\circ u^{-1}$, where we recall that $S_n =	S_{k, n}(z):= \frac{P^{(k)}_n(z)}{k!P_n(z)}$, yields the following inequality.
 \begin{proposition}\cite[Proposition 2.2]{angst2024almost}
 	For any Möbius transformation $u \in \mathcal{M}$ such that $u^{-1}(0)$ is not a zero or a pole of $S_n$, we have almost surely,
 	\begin{equation}\label{eq:Jensen:n}
 		\sum_{P_n^{(k)}(\rho)=0} \log^{-} |u(\rho)| - 	\sum_{P_n(\zeta)=0} \log^{-} |u(\zeta)|\le \log||S_n||_{u^{-1}(\mathbf S^1)} - \log |S_n|(u^{-1}(0))
 	\end{equation}
 	where $\mathbf S^1$ is the unit circle in $\C$ and $||S_n||_{u^{-1}(\mathbf S^1)}$ is the maximum value of $S_n$ over $u^{-1}(\mathbf S^1)$. Here, we note that the first sum runs over all zeros of $S_n$ while the second sum runs over all poles, both of which run with multiplicities.
 \end{proposition}
 
 Since the space of probability measures on $\C$ is a subset of the space $\mathcal P(\hat \C)$ of probability measures on the Riemann sphere $\hat \C = \C\cup \{\infty\}$ which is compact in the weak topology, it suffices to show that for any cluster value $\widehat{\mu}_{\infty}$ of the sequence $\mu_{P_n^{(k)}}$ in $ \mathcal P(\hat \C)$, we have $\widehat{\mu}_{\infty} = \mu$. Since both measures are probability measures on $\widehat{\mathbb C}$, 
 it suffices to show the upper bound $\widehat{\mu}_{\infty} \le \mu$, 
 which follows (by Lemma \ref{lm:logpotential} below) from proving that, for almost all Möbius transformations 
 $u \in \mathcal M$, the logarithmic potential of $\widehat{\mu}_{\infty}$ is dominated by that of $\mu$.
 \begin{lemma}\cite[Lemma 2.7]{angst2024almost}\label{lm:logpotential}
 	Let $\widehat m_1$ and $\widehat m_2$ be two finite measures on 
 	the Riemann sphere $\widehat{\mathbb C}$ such that
 	\[
 	\int_{\widehat{\mathbb C}} \log^{-}|u|\, d\widehat m_1
 	\le
 	\int_{\widehat{\mathbb C}} \log^{-}|u|\, d\widehat m_2
 	\]
 	for $\lambda_{\mathcal M}$–almost every Möbius transformation 
 	$u \in \mathcal M$. 
 	Then $\widehat m_1 \le \widehat m_2$ on $\widehat{\mathbb C}$. In particular, if $\widehat m_1$ and $ \widehat m_2$ are probability measures then they are equal.
 \end{lemma}
 
 Therefore, we are left to prove that almost surely with respect to $\P$, for $\lambda_{\mathcal{M}}$-almost every $u \in \mathcal{M}$, for all $\ep>0$,
 \begin{equation}
 	\int_{\widehat{\C}} \log^- |u|\, d\widehat{\mu}_\infty(u) \le \int_{\C} \log^- |u|\, d\mu(u) + \ep. \label{eq:mun:mu}
 \end{equation}
 Since $\widehat{\mu}_\infty$ is a cluster point, there exists a subsequence $\mu_{P_n^{(k)}}$ such that 
 $\widehat{\mu}_\infty = \lim_{k \to +\infty} \mu_{P_n^{(k)}}$.  
 We define the truncated map $x\mapsto \log^-_M(x) := \log^-(x) \wedge M$ which is continuous on $\widehat{\mathbb{C}}$, and therefore
 \[
 \int_{\widehat{\mathbb{C}}} \log^-_M |u| \, d\widehat{\mu}_\infty (u) 
 = \lim_{n \to \infty} \int_{\mathbb{C}} \log^-_M |u| \, d\mu_{P_n^{(k)}}(u)
 \le \limsup_{n \to \infty} \int_{\mathbb{C}} \log^- |u| \, d\mu_{P_n^{(k)}}(u).
 \]
 Sending $M \to \infty$, we deduce that
 \begin{equation}
 	\int_{\widehat{\mathbb{C}}} \log^- |u_i| \, d\widehat{\mu}_\infty
 	\le \limsup_{n \to \infty} \int_{\mathbb{C}} \log^- |u_i| \, d\mu_{P_n^{(k)}}(u). \label{eq:limsup}
 \end{equation}
 Moreover, by the Law of Large Numbers, as $\xi_i\sim \mu$ are iid,
 \begin{equation}
 	\int_{{\C}} \log^- |u|\, d\mu_n(u) = \frac1n\sum_{i=1}^{n} \log^-|u(\xi_i)|\to \int_{\C} \log^- |u|\, d\mu(u) \text{ almost surely}.  \notag
 \end{equation}
 Thus, \eqref{eq:mun:mu} is reduced to proving
 $$ \sum_{P_n^{(k)}(\rho)=0} \log^{-} |u(\rho)| - 	\sum_{P_n(\zeta)=0} \log^{-} |u(\zeta)|\le   \ep.$$
 Since for almost all Möbius transformations $u$, we have $\P$-a.s., $u^{-1}(0)$ is neither a pole or zero of $S_n$, we can apply \eqref{eq:Jensen:n} to further reduce the above to
 \begin{equation}\label{eq:S:u}
 	\limsup_{n\to \infty} \frac{\log||S_n||_{u^{-1}(\mathbf S^1)} - \log |S_n|(u^{-1}(0))}{n}\le \ep.
 \end{equation}
 
 This follows from the next two lemmas: the first proves a rather straightforward upper bound on $S_n$ and the second establishes our key anti-concentration result.
 \begin{lemma} \label{lm:upper}
 	For all probability measures $\mu$, for all $\ep>0$, for almost all $u\in \mathcal M$, it holds that $\P$-a.s., as $n \to +\infty$,
 	$$\log \|S_n\|_{u^{-1}(\mathbf S^1)} \le \ep n.$$
 \end{lemma}
 
 \begin{lemma} [Anti-concentration bound]\label{lm:lower}
 	Under the hypothesis of Theorem \ref{thm:main} \eqref{cond:nondeg}, for all $a\in \C$ and $\ep>0$, it holds that $\P$-a.s., as $n \to +\infty$,
 	$$\log |S_n(a)| \ge -\ep n.$$
 \end{lemma}
 \begin{proof}[Proof of Lemma \ref{lm:upper}] For $\lambda_{\mathcal{M}}$-almost every $u \in \mathcal{M}$, $u^{-1}(\mathbf S^1)$ is a circle $\mathbf{S}(a, r)$ where $a = u^{-1}(0)\in \C$ and $r\in (0, \infty)$.
 	We have
 	\begin{eqnarray*}
 		\|S_n\|_{\mathbf S(a, r)} &\le& {n\choose k} \sup_{i_1<\dots<i_k} \frac{1}{||a| - |\xi_{i_1}||}\dots  \frac{1}{||a| - |\xi_{i_k}||}\\
 		&\le& {n\choose k} \left (\sum_{i=1}^{n}\frac{1}{\left ||a| - |\xi_{i}|\right |^{1/2}}\right )^{2k}\\
 		&\le& {n\choose k} \left (O(n)\right )^{2k} = (O(n))^{2k}.
 	\end{eqnarray*}
 	Therefore, 
 	$$\log ||S_n||_{\mathbf S(a, r)} \ll k\log n\le \ep n$$
 	as desired.   
 \end{proof}
 Finally, we prove the anti-concentration bound.
 \begin{proof}[Proof of Lemma \ref{lm:lower}] By Borel-Cantelli Lemma, it suffices to show that there exists a sequence $p_n$ such that $\sum_n p_n<\infty$ and 
 	\begin{equation}\label{eq:smallball}
 		\P( |S_{k, n}(a)| \le  e^{-\ep n})\le p_n.
 	\end{equation}
 	
 	Since $\mu$ is dimension-nondegenerate, there exists a measure $\nu$ and a set $E$ as in Definition \ref{def:dim}. Note that there exists a constant $R>0$ for which $\nu(E\cap D(0, R))>0$. Therefore, replacing $E$ by $E\cap D(0, R)$ and $\nu$ by the measure $\nu$ restricted on $E\cap D(0, R)$ (which only makes the lower local dimension bigger), we can assume that $E\subset D(0, R)$ and $\nu(\C\setminus E) = 0$. Next, for each $q\in \mathbb Q_+$ and $r_0\in \mathbb Q\cap (0, 1)$, let
 	$$E_{q, r_0}:= \{x\in E: \nu(D(x, r))\le r^{q}, \ \forall r\in (0, 2r_0)\}.$$
 	If for $x\in E$, $\underline{\dim}_{\nu}(x)=s>0$ then  there exists $r_0<1$ such that for all $r < r_0$, it holds that
 	$$ \frac{\log \nu(D(x,r))}{\log r} \ge s/2$$
 	which implies 
 	$$  \nu(D(x,r) \le r^{s/2}\le r^q, \ \forall q\ge s/2.$$
 	Therefore, $$E  =\cup_{q, r_0} E_{q, r_0}.$$
 	By countability of the set of $(q, r_0)$, there exist $q$ and $r_0$ such that
 	$$\nu(E_{q, r_0})>0.$$
 	We then further replace $E$ by $E_{q, r_0}$ and $\nu$ by its restriction on this set and fix the choice of $q$ and $r_0$. With this choice, we have $\nu(E)>0$, $\nu(\C\setminus E) = 0$ and for all $x\in E$, for all $r<2r_0$, 
 	$$ \nu(D(x, r))\le r^{q}.$$
 	This implies that for all $x\in \C$, for all $r<r_0$,
 	 \begin{equation}
 	 	 \nu(D(x, r))\le r^{q}. \label{eq:xr}
 	 \end{equation}
 	We recall that $S_{k, n}(a)$ can be rewritten as a multivariate polynomial of degree $k$:
 	$$S_{k, n}:=S_{k, n}(a)= \sum_{i_1<i_2<\dots<i_k} Y_{i_1}\dots Y_{i_k}$$
 	where $Y_i = \frac{1}{a-\xi_i}$.
 	
 	The sampling of independent copies $\xi_i\sim \mu$ can be done in a way that separates the sources $\mu - \nu$ and $\nu$ as follows. Let $p = \mu(\C) - \nu(\C)\in [0, 1)$, we first sample independent copies $\eta_i\sim \text{Ber}(p)$. If $\eta_i=1$, we sample $\xi_i$ from $\frac{\mu-\nu}{p}$; otherwise, we sample $\xi_i$ from $\frac{\nu}{1-  p}$. Let $\mathcal I = \{i: \eta_i = 0\}$. Conditioning on $\eta_1, \dots, \eta_n$ and on $\xi_i$ with $i\notin \mathcal I$, the source of randomness remains is from $(\xi_i)_{i\in \mathcal I}$ which means they are sampled from the probability measure $\frac{\nu}{1-  p}$. 
 	
 	Let  $K = |\mathcal I|$. 
 	The anti-concentration bound is then reduced to this randomness:
\begin{eqnarray*}
 \P( |S_{k, n}(a)| \le  e^{-\ep n}) &=&\P(|S_{k, n}(a)|\le \gamma_n^k)  \\
	&\le&  \P(|S_{k, n}(a)|\le \gamma_n^k, K\ge (1-p)n/2) + \P(K< (1-p)n/2)\\
	&\le& \sup  \P_{(\xi_i)_{i\in \mathcal I}} (|S_{k, n}(a)|\le \gamma_n^k\big \vert \eta_1, \dots, \eta_n, (\xi_i)_{i\notin \mathcal I}) + e^{-cn}
\end{eqnarray*}
 where $\gamma_n  = e^{-\ep n/k}$ and the supremum runs over all possible realizations of $\eta_1, \dots, \eta_n, (\xi_i)_{i\notin \mathcal I}$ with $K\ge (1-p)n/2$. The upper-bound $e^{-cn}$ for $\P(K< (1-p)n/2)$ comes from Chernoff's inequality for some constant $c$.

In what follows, we condition on any realizations of $\eta_1, \dots, \eta_n, (\xi_i)_{i\notin \mathcal I}$ with $K\ge (1-p)n/2$.  Without loss of generality, we can assume that $\mathcal I = \{1, \dots, K\}$. Let
 	 $$\alpha_n=\sup_{w\in \C}\P\left (\left |\frac{1}{a - \xi}-w\right  |\le  \gamma_n\right )$$
 	 where $\xi\sim \nu/(1-p)$. 	 	 
 	 Since $$S_{k, n} = Y_1 S_{k-1, n-1}+S_{k-1, n}$$ with the $S_{k-1, n-1}$ and $S_{k-1, n}$ being independent of $Y_1$, we have for all $w\in \C$, 
 	 $$\P(|S_{k, n} - w|\le \gamma_n^{k}) \le \P(|S_{k, n} - w|\le \gamma_n^{k}, |S_{k-1, n-1}|\ge \gamma_n^{k-1}) +  \P(|S_{k-1, n-1}|\le \gamma_n^{k-1}).$$
 	 The first term on the right is bounded by
 	 $$\P\left (|Y_1+\frac{S_{k-1, n}-w}{S_{k-1, n-1}}|\le \frac{\gamma_n^{k}}{S_{k-1, n-1}}, |S_{k-1, n-1}|\ge \gamma_n^{k-1}\right ) \le \P\left (|Y_1+\frac{S_{k-1, n}-w}{S_{k-1, n-1}}|\le \gamma_n\right ) \le \alpha_n,$$
 where the last inequality follows	by the independence of $Y_1$ and $\frac{S_{k-1, n}-w}{S_{k-1, n-1}}$.
 	 Therefore, 
 	 $$\P(|S_{k, n} - w|\le \gamma_n^{k}) \le \alpha_n +  \P(|S_{k-1, n-1}|\le \gamma_n^{k-1}).$$
 Since $S_{k-1, n-1}$ has the same form as $S_{k, n}$ except with one degree smaller, we can repeating this multiple times to get
 	 $$\P(|S_{k, n} - w|\le \gamma_n^{k}) \le (k-1)\alpha_n +\sup_{w'\in \C}\P(|Y_k+\dots +Y_K+w'|\le \gamma_n) \le k\alpha_n$$
 	 where in the last inequality, we again use the independence between $Y_k$ and the rest of the $Y_i$, noting that $K=\Theta(n)\gg k$.

 	 Thus, it suffices to show that $\sum_{n} k \alpha_n <\infty$. To this end, for any fixed $a\in \C$, let $\pi$ be the distribution of $Y = \frac{1}{a - \xi}$ where $\xi\sim \frac{\nu}{1-p}$. Let $\nu_a$ be the distribution of $a - \xi$ which is supported inside the disk of radius $R_a  = R + |a|$ around 0. Let $T$ be the inversion map $z\mapsto z^{-1}$. Note that $T$ maps $D(0, R_a)$ into $\C\setminus \text{int}(D (0, r_a))$ some radius $r_a>0$. Then for all $w\in \C$, 
 	 $$\P\left (\left |\frac{1}{a - \xi}-w\right  |\le  \gamma_n\right )=\pi(D(w, \gamma_n)) = \nu_a(T^{-1}(D(w, \gamma_n)))=\nu_a(T(D(w, \gamma_n))).$$
 	 Let $c' = r_a/4$. If $D(w, c')$ contains $0$ then $T(D(w, \gamma_n))\subset T(D(w, c'))\subset T(D(0, 2c'))$ which is the outer circle of radius larger than $R_a$ for sufficiently large $n$, and hence  $\nu_a(T(D(w, \gamma_n)))=0$. 
 	 
 	 Therefore, we can assume that $0\notin D(w, c')$, which implies that $|w|\ge c'\gg \gamma_n$. In that case, the (pre)image of $D(w, \gamma_n)$ under $T$ is a circle of radius
 	 $$r = \frac{\gamma_n}{||w|^2 - \gamma_n^2|}\le \frac{ \gamma_n}{c'^2/2}.$$
 	 And then for all sufficiently large $n$ so that $\frac{ \gamma_n}{c'^2/2}<r_0$, we get by \eqref{eq:xr},
 	 $$k\pi(D(w, \gamma_n)) \le k\nu_a(D(w', \frac{ \gamma_n}{c'^2/2}))\ll k \gamma_n^{q}.$$
 	 For $k = o(\frac{n}{\log n})$, we have
 	 $$\sum_{n} k \gamma_n^{q} = \sum_{n} k e^{-cqn/k}<\infty.$$
 	 This completes the proof.
 \end{proof}

 \section{Proof of Theorem \ref{thm:2}}
 \subsection{Doeblin condition}
   As before, we reduce to proving the anti-concentration Lemma \ref{lm:lower} (and in particular, \eqref{eq:smallball}) for the case that $\mu$ satisfies Condition \eqref{cond:0} in Theorem \ref{thm:2}.  
 Assume for the moment that $Y_i$ were uniform on some disk $B\subset \C$ of area 1. We recall Carbery-Wright inequality \cite[Theorem 8]{carbery2001distributional}.
 \begin{theorem}
 	There exists an absolute constant $C$ such that if $p: \R^{n}\to \C$ is a polynomial of degree at most $k$ and $\pi$ is a log-concave measure on $\R^{n}$ then for all $\alpha>0, q>0$,
 	\begin{equation}\label{eq:CW}
 		\pi(x\in \R^{n}: |p(x)|\le \alpha)\le \frac{Cq\alpha^{1/k}}{\left (\int |p(x)|^{q/k}d\pi (x)\right )^{1/q}}
 	\end{equation}
 \end{theorem}
 We apply this theorem to $p=S_{n}$, $\alpha =e^{-\ep n}$, $q = 2k$ and $\pi=\text{Unif}(B)$. 
 To evaluate the second moment, we write $c = \E Y_i$, $\overline Y_i = Y_i - c$ and 
 $$S_n=\sum_{i_1<\dots<i_k} \prod_{j=1}^{k} (\overline Y_{i_j}+c)=:\sum_{I\subset [n]} \alpha_I\prod_{i\in I} \overline Y_i$$
 where the sum runs over subsets of indices of cardinality at most $k$ and $\alpha_I\in \C$ are deterministic coefficients. We note that for $I$ with cardinality $k$, $\alpha_I = 1$.
 We have by independence and the fact that $\E \overline Y_i = 0$, 
 \begin{eqnarray*}
 	\E |S_n|^2 &=& \sum_{I\subset [n]} |\alpha_I|^{2} \E |\overline Y_1|^{2} + \sum_{I\neq J} \alpha_I \overline \alpha_J \E (\prod_{i\in I} \overline Y_i \prod_{i\in J} \overline Y_i) \notag\\
 	&=& \sum_{I\subset [n]} |\alpha_I|^{2} \E |Y_1|^{2} \ge \sum_{|I| = k}  \E |Y_1|^{2} = {n\choose k} \E |Y_1|^2.
 \end{eqnarray*}
 Thus, 
 \begin{eqnarray*}
 	\left (\int |p(x)|^{2}d\phi (x)\right )^{1/(2k)}&=& \left (\E |S_n|^{2}\right )^{1/(2k)}	\ge   C {n \choose k}^{1/2k}.
 \end{eqnarray*}
 By Stirling's formula, for $k = o(n)$, this is of order 
 $$\Theta\left (\sqrt {\frac{n}{k}} \left (1 + \frac{k}{n-k}\right )^{(n-k)/(2k)}\right ) = \Theta(1)\sqrt{\frac{n}{k}}.$$
 So then, the right-hand of \eqref{eq:CW} becomes
 $$\frac{Ck^{3/2}}{e^{\ep n/k}\sqrt n}\le \frac{1}{n^2},$$
 for sufficiently large $n$, as $k=o( \frac{n}{\log n})$. Therefore, 
 $$\P_{Y_i\overset{iid}{\sim} \text{Unif}(B)}\left ( \left |\sum_{i_1<\dots<i_k} Y_{i_1}\dots Y_{i_k}\right | \le  e^{-\ep n}\right )\le n^{-2}.$$

 In the general case where $Y_i$ is not uniform on a disk, we use the following lemma, known as Nummelin's splitting allows us to decompose $Y_i$ into uniform random variables and the rest. This decomposition, in the real setting, has been used in several places such as \cite{poly2012dirichlet}, \cite{nourdin2015invariance}, \cite{BallyCaramellino2019} and references therein. Because we could not locate a reference for the complex case, we provide a proof of the lemma below.
 
 \begin{lemma}\label{lm:splitting}
 	Assume that the measure $\mu$ satisfies the Doeblin's condition \ref{cond:doeblin}, for all $z\neq z_0$, there exist parameters $c_a\in (0, 1), r_a>0, w_a\in \C$ such that  
 	$$Y_i = \frac{1}{z-\xi_i} \overset{d}{=} \ep_i Z_i + (1-\ep_i)W_i$$
 	where $\ep_i, Z_i, W_i$ are independent random variables with $\ep_i\sim \text{Ber}(c_a)$, $Z_i\sim \text{Unif}(B(w_a, r_a))$.
 \end{lemma}
 Applying this lemma, we can then write $S_n$ in terms of the $\ep_i, Z_i, W_i$.
 Conditioned on the $\ep_i$ and $W_i$, $S_n$ is a polynomial of degree $k$ on the random variables $Z_i$ which can be written as
 $$S_n = \sum_{i_1<\dots<i_k,\, i_j\in \mathcal I} Z_{i_1}\dots Z_{i_k}+T_n$$
 where $\mathcal I = \{i\in [n]: \ep_i=1\}$ and $T_n$ is a polynomial of degree $k-1$ in the $Z_i$.
 Thus, the second moment of $S_n$ with respect to $Z_i$ is at least $C{{|\mathcal I|}\choose k}$ as before, for some constant $C$.
 By Carbery-Wright inequality applying to the polynomial $S_n$ of uniform random variables $Z_i$, we have
 $$\P_{Z_i}(|S_n|\le e^{-\ep n})\le \frac{C ke^{-\ep n/k}}{{{|\mathcal I|}\choose k}^{1/(2k)}}.$$
 By Chernoff's inequality, with probability at least $e^{-c_{a}n}$, we have $|\mathcal I|\ge c_an/2$. Thus, 
 $$\P(|S_n|\le e^{-\ep n})\le e^{-c_{a}n}+ c_a^{-2}n^{-2}$$
 which remains summable. This completes the proof of Lemma \ref{lm:lower}.
 
 \begin{proof} [Proof of Lemma \ref{lm:splitting}]
 	Let
 	\[
 	Y := \frac{1}{a - \xi}, \qquad \xi \sim \mu.
 	\]
 	We first show that $Y$ also satisfies the Doeblin's condition, possibly with different parameters. Indeed, set $\delta := |a - z_0|>0$ and choose
 	\[
 	r_a := \min\{r_0, \delta/2\}.
 	\]
 	Then for every $w \in B(z_0, r_a)$ we have $|a - w|\ge \delta/2.$
 	Define
 	\[
 	T_a : B(z_0, r_a) \to \C, \qquad T_a(w) := \frac{1}{a - w}.
 	\]
 	The map $T_a$ is holomorphic on $B(z_0, r_a)$, with derivative
 	\[
 	T_a'(w) = \frac{1}{(a-w)^2}.
 	\]
 	Since on $B(z_0,r_a)$ we have
 	\[
 	\delta/2 \le |a-w| \le |a - z_0| + r_a \le \delta + \delta/2 = 3\delta/2,
 	\]
 	we obtain that $ |T_a'(w)|$ is uniformly bounded above and below. Therefore, $T_a$ is a $C^1$-diffeomorphism from $B(z_0, r_a)$ onto the open set
 	\[
 	D_a := T_a\big(B(z_0, r_a)\big).
 	\]
 	Moreover, the Jacobian of the inverse $T_a^{-1}$ is
 	\[
 	J_{T_a^{-1}}(y) = \frac{1}{J_{T_a}(T_a^{-1}(y))},
 	\]
 	so it is also bounded above and below on $D_a$.
 	
 	\medskip
 	
 	Let $\nu_a$ be the distribution of $Y$, then for any Borel set $E \subset D_a$, we have $T_a^{-1}(E) \subset B(z_0, r_a) \subset B(z_0, r_0)$, so from \eqref{cond:doeblin},
 	\begin{equation}
 		\nu_a(E) = \mu\big(T_a^{-1}(E)\big) \ge c_0 \, \mathrm{Leb}\big(T_a^{-1}(E)\big).
 		\label{eq:nu-lower}	 		
 	\end{equation}
 	Since $T_a$ is a diffeomorphism on this set and $J_{T_a^{-1}}$ is bounded below there exists $m_a>0$ (depending only on $a, z_0, r_0, c_0$) such that
 	\begin{equation}
 		\mathrm{Leb}\big(T_a^{-1}(E)\big)
 		= \int_{E} J_{T_a^{-1}}(y)\, dy
 		\ge m_a \int_{E} dy
 		= m_a \, \mathrm{Leb}(E)
 		\qquad \text{for all Borel } E \subset D_a.
 		\label{eq:jac-lower}
 	\end{equation}
 	Combining \eqref{eq:nu-lower} and \eqref{eq:jac-lower} we get
 	\begin{equation}
 		\nu_a(E) \ge c_0 m_a \, \mathrm{Leb}(E)
 		\qquad \text{for all Borel } E \subset D_a.
 		\label{eq:doeblin-Y-a}
 	\end{equation}
 	
 	\medskip
 	
 	Since $D_a$ is open, there exists a ball $B(w_a, r_a) \subset D_a$.  
 	Then from \eqref{eq:doeblin-Y-a},
 	\[
 	\nu_a(E) \ge c_0 m_a \, \mathrm{Leb}(E)
 	\qquad \text{for all Borel } E \subset B(w_a, r_a).
 	\]
 	Thus, $Y$ satisfies the Doeblin's condition. We then set
 	\[
 	c_a := c_0 m_a \mathrm{Leb}\big(B(w_a, r_a)\big) \le \nu_a(B(w_a, r_a)).
 	\]
 	Thus, $c_a\in (0, 1]$.
 	If $c_a=1$ the decomposition is trivial; otherwise we define the uniform measure
 	\[
 	\lambda_a(E) := \frac{\mathrm{Leb}(E \cap B(w_a, r_a))}{\mathrm{Leb}(B(w_a, r_a))}.
 	\]
 	Then for every $E \subset B(w_a, r_a)$,
 	\begin{equation}
 		\nu_a(E) \ge c_a \, \lambda_a(E).
 	\end{equation}
 	
 	\medskip
 	
 	Define the probability measure
 	\[
 	\eta_a := \frac{\nu_a - c_a \lambda_a}{1 - c_a}.
 	\]
 	Then we have the decomposition
 	\[
 	\nu_a = c_a \lambda_a + (1-c_a)\eta_a.
 	\]
 	Equivalently, we can realize $Y$ as
 	\[
 	Y \overset{d}{=} \ep Z + (1-\ep) W,
 	\]
 	where $\ep  \sim \mathrm{Ber}(c_a)$, $Z \sim \lambda_a$ , $W \sim \eta_a$, and $\delta_a, Z, W$ are independent.
 \end{proof}

 \subsection{Generic measures} 
  It suffices to prove \eqref{eq:smallball2}. For almost all $a\in \C$, by the hypothesis of Theorem \ref{thm:2} \eqref{cond:mean}, we know that
 $\E Y_i=\int_{\C} \frac{1}{a-u}d\mu(u)=: c\neq 0$
 where $c$ depends only on $a$ and $\mu$. We then have
 $$|\E S_n| = |c|^{k}{n\choose k}\to \infty\text{ as } n\to\infty \text{ for $k=o(n)$}.$$
 To calculate higher moments of $S_n$, we first normalize the  $Y_i$ as $Y_i  = \sigma X_i +c$ where $\sigma = \sqrt{\Var(Y_i)}$. The random variables $X_i$ are independent with mean 0, variance 1. We have
 \begin{eqnarray*}
 	T_n:= S_n -\E S_n &=& \sum_{i_1<\dots<i_k}  \sum_{r=1}^{k} \sum_{j_1<\dots<j_r \text{ subset of } i_1<\dots <i_k}c^{k-r} \sigma^{r} X_{j_1}\dots X_{j_r}\\
 	&=& \sum_{r=1}^{k} c^{k-r} \sigma^{r} {n\choose {k-r}}\sum_{j_1<\dots<i_r} X_{j_1}\dots X_{j_r}
 \end{eqnarray*}
 Therefore,
 \[
 \mathbb E[|T_n|^2]=\sum_{r=1}^k N_r\, \sigma^{2r} |c|^{2k-2r}.
 \]
 where
 \[
 N_r=\binom{n}{r}\binom{n-r}{k-r}^{2}.
 \]
 The quotient between two consecutive terms in the summation is 
 \begin{eqnarray*}
 	\frac{N_{r} \sigma^2}{N_{r-1}|c|^{2}} = \frac{(k-r+1)^{2}\sigma^{2}}{(n-r+1)r|c|^{2}}=o_{a}(1)
 \end{eqnarray*}
 as $k = o(\sqrt n)$. And so,
 \begin{equation}
 	\E |T_n|^{2}\ll N_1|c|^{2k-2}\sigma^2 = \frac{k^{2}\sigma^2}{n|c|^{2}} |\E S_n|^{2}.\label{eq:second:moment}
 \end{equation}
 Therefore, as $k = o(\sqrt n)$, we have by Chebyshev's inequality,
 $$\P(\log |S_n(a)| \le -\ep n) \le \P(|S_n(a) - \E S_n(a)|\ge |\E S_n(a)|/2)\ll \frac{\Var S_n(a)}{|\E S_n(a)|^{2}}=o(1)$$
 as desired and concluding the proof.

 \section{Proof of Theorem \ref{thm:perturb}}
 We will adapt the proof in Section \ref{sec:proof:mu}. By Lemma \ref{lm:logpotential}, it would be desirable to show that for all $\ep>0$ and all cluster value $\hat \mu_n$ of the sequence $\mu_{P_n^{(k)}}$ in $ \mathcal P(\hat \C)$, we have 
 \begin{equation}
 	\int_{\widehat{\mathbb C}} \log^{-}|u|\, d\widehat \mu_\infty \le \int_{\widehat{\mathbb C}} \log^{-}|u|\, d\mu + \ep \notag
 \end{equation}
 for $\lambda_{\mathcal M}$–almost every Möbius transformation  $u \in \mathcal M$. 
 However, the left-hand side contains some contribution from $\int_{\widehat{\mathbb C}} \log^{-}|u|\, d \nu$ which may be infinite. To avoid this issue, we instead show that for all $\ep, \ep'>0$, for almost all $u\in \mathcal M$,
 \begin{equation}
 	\int_{\widehat{\mathbb C}} \log^{-}|u|\, d\widehat \mu_\infty \le \int_{\widehat{\mathbb C}} \log^{-}|u|\, d(\mu+\ep \nu) + \ep.\label{eq:mun:mu:2}
 \end{equation}
 By Lemma \ref{lm:logpotential}, we get $\widehat \mu_\infty \le \mu+\ep \nu$ for all $\ep$ which in turns shows that $\widehat \mu_\infty \le \mu$ and hence $\widehat \mu_\infty = \mu$ as they are probability measures.

 To prove \eqref{eq:mun:mu:2}, as we have shown in \eqref{eq:limsup},
 \begin{equation}
 	\int_{\widehat{\mathbb{C}}} \log^- |u_i| \, d\widehat{\mu}_\infty
 	\le \limsup_{n \to \infty} \int_{\mathbb{C}} \log^- |u_i| \, d\mu_{P_{n}^{(k)}}(u). \label{eq:nu1}
 \end{equation}
 Next, observe that a way to sample $\xi_i\sim \mu_n$ is by first sampling $\theta_i\sim \text{Ber}(\alpha_n)$, $\zeta_i\sim \mu$, $\eta_i\sim \nu$ and then setting
 $$\xi_i = \zeta_i \textbf{1}_{\theta_i = 0} + \eta_i \textbf{1}_{\theta_i=1}.$$
 We then get
 \begin{eqnarray*}
 	\frac1n\sum_{i=1}^{n} \log^-|u(\xi_i)| &=& \frac1n\sum_{i: \theta_i=0} \log^-|u(\zeta_i)| + \frac1n\sum_{i: \theta_i=1} \log^-|u(\eta_i)|\\
 	&\le& \frac1n\sum_{i=1}^{n} \log^-|u(\zeta_i)| + \frac{N_n}{n}\frac{1}{N_n}\sum_{i: \theta_i=1} \log^-|u(\eta_i)|
 \end{eqnarray*}
 where $N_n\sim \text{Bin}(n, \alpha_n)$ is the number of indices $i$ for which $\theta_i = 1$.
 The first sum, as seen before, by the Law of Large Numbers,
 \begin{equation}
 	\frac1n\sum_{i=1}^{n} \log^-|u(\zeta_i)| \to \int_{\mathbb{C}} \log^- |u_i|d\mu \text{ almost surely.}\label{eq:nu2}
 \end{equation}

 We claim that
 \begin{equation}
 	N_n\to \infty \text{ almost surely}\label{eq:Nn1}
 \end{equation}
 and
 \begin{equation}
 	\frac{N_n}{n}\to 0 \text{ almost surely}\label{eq:Nn2}.
 \end{equation}
 To prove \eqref{eq:Nn1}, we apply the Chernoff lower-tail bound:  
 \[
 \P (N_n \le n \alpha_n/2 )\le \exp\!\left(- \frac18\,n \alpha_n\right).
 \]
 Since $n \alpha_n /\log n\to \infty$, the right-hand side is at most $n^{-2}$ for sufficiently large $n$ and hence $\sum_n \P(N_n \le n\alpha_n/2) < \infty$. By the Borel--Cantelli lemma,
 $\P(N_n \le n\alpha_n/2 \text{ i.o.}) = 0$. It follows that $N_n \to \infty$ almost surely.
 
 To prove \eqref{eq:Nn2}, for any $\delta > 0$, since $a_n \to 0$, there exists $n_0$ such that
 $a_n \le \delta/2$ for all $n \ge n_0$. For such $n$ we use the standard
 Chernoff bound for binomial variables:
 \[
 \P(N_n \ge \delta n)
 \le \left( \frac{e n a_n}{\delta n} \right)^{\delta n}
 = \left( \frac{e a_n}{\delta} \right)^{\delta n}.
 \]
 Since $a_n \to 0$, the right-hand side is summable in $n$, hence
 $\sum_{n} \P(N_n \ge \delta n) < \infty$. By the Borel--Cantelli lemma,
 $\P(N_n \ge \delta n \text{ i.o.}) = 0$. As this holds for
 $\delta = 1, 1/2, 1/3, \dots$, we conclude that $N_n/n \to 0$ almost surely.
 
 Now, using \eqref{eq:Nn1}, we obtain by the Law of Large Numbers that $\sum_{i: \theta_i=1} \log^-|u(\eta_i)|$ converges almost surely to its mean $ \int_{\C} \log^- |u|\, d\nu$, which by combining with \eqref{eq:Nn2}, gives
 \begin{equation}
 	\frac{N_n}{n}\frac{1}{N_n}\sum_{i: \theta_i=1} \log^-|u(\eta_i)|\le \ep \int_{\C} \log^- |u|\, d\nu \text{ \ almost surely eventually}.\label{eq:nu3}
 \end{equation}
 Combining \eqref{eq:nu1}, \eqref{eq:nu2}, and \eqref{eq:nu3}, we reduce the problem of proving \eqref{eq:mun:mu:2} to, as before, proving the corresponding version of \eqref{eq:S:u}, namely almost surely
 \begin{equation}
 	\limsup_{n\to \infty} \frac{\log||S_n||_{u^{-1}(\mathbf S^1)} - \log |S_n|(u^{-1}(0))}{n}\le \ep.\notag
 \end{equation}
 where
 $S_n = \frac{P_n^{(k)}}{k!P_n} = \sum_{i_1<\dots<i_k} Y_{i_1}\dots Y_{i_k}$ where $Y_{i_j} = \frac{1}{z - \xi_i}$. By the same argument, Lemma \ref{lm:upper} continues to hold. As a final step, we adapt the proof of Lemma \ref{lm:lower} by exploring the randomness of the $\eta_i$. 
 
 More specifically, we first sample the random variables $(\theta_i)$ and $(\zeta_i)$. This will determine the set $\mathcal I=\{i: \theta_i=1\}$ and its cardinality $N$. As we have proved before, almost surely, $N\ge n\alpha_n/2$ for all $n\ge n_0$ for some $n_0$ only determined by the values of $(\theta_i)$. Then, $S_n$ is now a degree-$k$ polynomial of $N$ iid random variables $\eta_i\sim \nu$, $i\in \mathcal I$. 
 
 Note that $S_n$ has the form
 $$S_n  = \sum_{I\subset \mathcal I} \alpha_I \prod_{i\in I} Y_i$$
 where the sum runs over all subsets $I$ of cardinality at most $k$ and for all $|I| = k$, we have $\alpha_I = 1$. Therefore, the rest of the proof of Lemma \ref{lm:lower} goes through as we replace $n$ by $N$ and $[n]$ by $\mathcal I$ giving us
 $$\P_{(\eta_i)}(|S_n|\le e^{-\delta n})\le \P_{(\eta_i)}(|S_n|\le e^{-\delta N}) \le  k e^{-q\delta N}+ e^{-c' N}\le k e^{-q\delta N} + e^{-c'\alpha_n n}\le n^{-2}$$
 as $k = o(\frac{\alpha_n n}{\log n})$. 
 By Borel-Cantelli lemma, almost surely, $|S_n|\le e^{-\delta n}$ eventually. This completes the proof.

 \bibliographystyle{alpha}

 \end{document}